\documentclass[10pt,a4paper]{article}

\newif\ifdraft 
\draftfalse    

\usepackage{latexsym}
\usepackage{amssymb}
\usepackage{amsmath}
\usepackage{amsthm}
\usepackage{thmtools}
\usepackage{enumitem}
\usepackage{cite}
\usepackage[all]{xy}
\usepackage{framed}
\usepackage{hyperref}
\usepackage{graphicx}
\usepackage{parskip}
\usepackage{hyperref}
\usepackage{bm}
\usepackage[a4paper,margin=2cm]{geometry}
\usepackage[utf8]{inputenc}
\usepackage{mathdots}

\begingroup
    \makeatletter
    \@for\theoremstyle:=definition,remark,plain\do{%
        \expandafter\g@addto@macro\csname th@\theoremstyle\endcsname{%
            \addtolength\thm@preskip\parskip
            }%
        }
\endgroup

\makeatletter
\xydef@\txt@ii#1{\vbox{\vspace*{-5pt}%
 \let\\=\cr
 \tabskip=\z@skip \halign{\relax\hfil\txtline@@{##}\hfil\cr\leavevmode#1\crcr}}}
\makeatother

\theoremstyle{definition}
\newtheorem{thm}{Theorem}[section]
\newtheorem{lem}[thm]{Lemma}
\newtheorem{cor}[thm]{Corollary}
\newtheorem{defn}[thm]{Definition}
\newtheorem{propn}[thm]{Proposition}
\newtheorem*{thm*}{Theorem}

\newtheorem{qn}[thm]{Question}
\newtheorem*{qn*}{Question}

\newtheorem*{nts}{Note to self}

\theoremstyle{remark}
\newtheorem{rk}[thm]{Remark}
\newtheorem{rks}[thm]{Remarks}
\newtheorem{ex}[thm]{Example}

\newtheorem{nonex}[thm]{Non-example}

\newtheoremstyle{custthm}{\parskip}{}{\normalfont}{}{\bfseries}{.}{ }{\thmname{#1} \thmnote{#3}}
\theoremstyle{custthm}

\newcommand{\tensor}[1]{\underset{#1}{\otimes}}

\newcommand{\gr}{\mathrm{gr}}

\newcommand{\Aut}{\mathrm{Aut}}

\newcommand{\Kdim}{\mathrm{Kdim}}
\newcommand{\clKdim}{\mathrm{clKdim}}
\newcommand{\gldim}{\mathrm{gldim}}
\newcommand{\Ext}{\mathrm{Ext}}

\newcommand{\pd}{\mathrm{projdim}}

\newcommand{\SPS}{\mathbf{SPS}}
\newcommand{\RSPS}{\mathbf{RSPS}}

{%
\begin{framed}\begin{nts}%
}%
{%
\end{nts}\end{framed}%
}

\ifdraft
 {%
 \end{framed}%
 }%
\fi

\begin{document}

\numberwithin{equation}{section}
\binoppenalty=\maxdimen
\relpenalty=\maxdimen

\title{Dimension theory in iterated local skew power series rings}
\author{William Woods}
\date{\today}
\maketitle
\begin{abstract}
By analysing the structure of the associated graded ring with respect to certain filtrations, we deduce a number of good properties of iterated local skew power series rings over appropriate base rings. In particular, we calculate the Krull dimension, prime length and global dimension in well-behaved cases, obtaining lower bounds to complement the upper bounds obtained by Venjakob and Wang.
\end{abstract}

\tableofcontents

%
%
%
%
%
%

\newpage

\section*{Introduction}

The main objects of study of this paper are \emph{iterated skew power series rings}
$$R = k[[x_1;\sigma_1, \delta_1]]\dots [[x_n; \sigma_n,\delta_n]]$$
for suitable choices of $k$ and $(\sigma_i,\delta_i)$. By this, we mean that $R_0 = k$, and we iteratively define for each $1\leq i\leq n$
\begin{equation}\label{eqn: iteration procedure}
\begin{cases}
\text{$(\sigma_i,\delta_i)$ is a \emph{skew derivation} on $R_{i-1}$,}\\
\text{$R_i$ is the \emph{skew power series ring} $R_{i-1}[[x_i;\sigma_i,\delta_i]]$,}
\end{cases}
\end{equation}
and finally set $R = R_n$. Full definitions of these objects are given below.

Skew power series extensions $B = A[[x;\sigma,\delta]]$ were first introduced in \cite{schneider-venjakob-codim} and \cite{venjakob}, with a view to using them to study Iwasawa algebras: it was noted in \cite{schneider-venjakob-codim} that certain Iwasawa algebras can be written as \emph{iterated} skew power series extensions over a base field (usually $\mathbb{F}_p$) or complete discrete valuation ring (usually $\mathbb{Z}_p$).

Modelled on the notion of a skew polynomial extension $A[x;\sigma,\delta]$, a \emph{skew power series} extension $A[[x;\sigma,\delta]]$ is a ring which is isomorphic to $A[[x]]$ as topological left (or right) $A$-modules, with multiplication given by
\begin{equation}\label{eqn: multiplication}
xa = \sigma(a)x + \delta(a)
\end{equation}
for all $a\in A$. Here, $\sigma$ is an automorphism of $A$, and $\delta$ is a \emph{$\sigma$-derivation} of $A$, i.e. a linear map $\delta: A\to A$ satisfying $\delta(ab) = \delta(a)b + \sigma(a)\delta(b)$. (We say that $(\sigma,\delta)$ is a \emph{skew derivation} for short.)

Unlike the case of skew polynomial extensions, it is \emph{not} the case that an arbitrary skew derivation $(\sigma,\delta)$ on $A$ gives rise to a well-defined ring $A[[x;\sigma,\delta]]$, due to possible convergence issues. For this reason, we work throughout the paper with a special class of skew power series extensions whose question of existence is already well understood. We denote this class of rings by $\SPS(A)$, and call them \emph{local skew power series extensions}: we require $A$ to have a unique maximal ideal $\mathfrak{m}$, with respect to which it is complete (and separated), and we stipulate that $\sigma(\mathfrak{m}) \subseteq \mathfrak{m}, \delta(R)\subseteq \mathfrak{m}$ and $\delta(\mathfrak{m})\subseteq \mathfrak{m}^2$. These are known to exist, and while other classes have been studied (see e.g. \cite{bergen-grzeszczuk-skew} and \cite[\S 5]{greenfeld-smoktunowicz-ziembowski}, dealing with the rather different case in which $\delta$ is assumed \emph{locally nilpotent}), the local case encompasses the case of interest in the Iwasawa context.

There is a growing literature on simple power series extensions $S = R[[x;\sigma,\delta]]$: see, for instance, \cite{schneider-venjakob-localisations} and \cite{letzter-noeth-skew} for important milestones in this theory. However, very little work has been done to treat the \emph{iterated} case so far. The most basic noncommutative case, that of \emph{$q$-commutative power series rings} (completed quantum planes), has been studied in \cite{letzter-wang-q-comm}. A few properties of general iterated local skew power series rings are established in \cite{wang-quantum} (though see the remarks before Theorem D below for a caveat). The author is not aware of any other developments in the theory.

As above, we restrict our work to the local case, and denote the class of $n$-fold iterated local skew power series extensions of $A$ by $\SPS^n(A)$ (Definition \ref{defn: SPS^n} below). Even under rather general conditions like this, we can understand much of the dimension theory of these rings: our first result, along these lines, complements the work of Wang \cite{wang-quantum}.

\textbf{Theorem A.} Let $k$ be a field, and $R\in\SPS^n(k)$. Then $\Kdim(R) = \gldim(R) = n$.

This result follows from Theorem \ref{thm: Kdim}, Theorem \ref{thm: clKdim} and Corollary \ref{cor: gldim}.

The process of \emph{iteration} is often rather badly behaved. Indeed, in the notation of (\ref{eqn: iteration procedure}): when $n \geq i \geq 2$, there is no guarantee that $(\sigma_i,\delta_i)$ restricts to a skew derivation of $R_j$ for any $0 \leq j \leq i-2$. Hence we identify a subclass of $\SPS^n(A)$, which we denote $\RSPS^n(A)$, and whose elements we call \emph{rigid}, whose elements are well-behaved in the above sense. See Definition \ref{defn: stability} for the precise definition.

This appears to be a fairly restrictive condition: we expect that $\RSPS^n(A)$ is in fact a much smaller class than $\SPS^n(A)$ except in trivial cases when $n$ and $A$ are very small. But, in fact many well-known rings \emph{do} lie in $\RSPS^n(A)$, including $q$-commutative skew power series rings (Example \ref{ex: q-commutatives}), Iwasawa algebras of supersoluble uniform groups (Examples \ref{ex: nilpotent IAs}--\ref{ex: rigid soluble IAs}), and various other completed quantum algebras (Examples \ref{ex: horton's algs}--\ref{ex: quantum matrices}).

We are interested in the prime ideal structure of iterated skew power series rings $R\in\SPS^n(A)$. As a first step towards a more general theory, we calculate the classical Krull dimension (prime length), as defined in \cite{MR}, of elements of $\RSPS^n(A)$, and show that they obtain the upper bound given by the Krull dimension, similarly to pure automorphic extensions, which are well studied and in some senses easier to understand. This perhaps indicates that, in understanding the theory of prime ideals of iterated non-pure-automorphic skew power series rings, a good first step will be to consider the \emph{rigid} ones.

\textbf{Theorem B.} Let $k$ be a division ring, and $R\in\SPS^n(k)$. If $R\in\RSPS^n(k)$, or $R$ has pure automorphic type, then $\clKdim(R) = n$.

The proof of Theorem B is contained within Theorem \ref{thm: clKdim} and Proposition \ref{propn: clKdim}.

So far we have not said much about the natural filtration on $R$, but this will be crucial for understanding its ideals in future work. If $R$ is an iterated local skew power series ring with unique maximal ideal $\mathfrak{m}$, we can consider $R$ as a complete $\mathfrak{m}$-adically filtered ring. We compute its associated graded ring as follows:

\textbf{Theorem C.} Let $(R, \mathfrak{m})$ be a complete local ring, $(\sigma,\delta)$ a local skew derivation on $R$, and $S = R[[x;\sigma,\delta]]$ a local skew power series extension of $R$ with unique maximal ideal $\mathfrak{n}$. Then there exists a skew derivation $(\overline{\sigma}, \overline{\delta})$ of $\gr_{\mathfrak{m}}(R)$ such that the inclusion of graded rings $\gr_{\mathfrak{m}}(R)\to \gr_{\mathfrak{n}}(S)$ extends to an isomorphism $(\gr_{\mathfrak{m}}(R))[X;\overline{\sigma},\overline{\delta}]\cong \gr_{\mathfrak{n}}(S)$ upon mapping $X$ to $\gr(x)$. Moreover, if $\delta(\mathfrak{m}) \subseteq \mathfrak{m}^3$, then $\overline{\delta} = 0$.

The proof is given in Lemma \ref{lem: isom of graded rings}.

We give Example \ref{ex: nonzero graded delta} to show that $\overline{\delta}$ can indeed be nonzero: this corrects a small error in \cite[Proposition 2.9]{venjakob}, \cite[2.3(iii)]{wang-quantum} and \cite[3.10(i)]{letzter-noeth-skew}, and a similar error can be found in \cite[Lemma 1.4(iv)]{schneider-venjakob-codim}. This phenomenon also arises in practice when studying the Iwasawa algebras of certain non-uniform 2-valuable groups. However, we can get rid of $\overline{\delta}$ by modifying the filtration, which we prove as Lemma \ref{lem: gr under new filtration}:

\textbf{Theorem D.} Let $R$ and $S$ be as in Theorem C. Then there exists a filtration $f$ of $S$, cofinal with the $\mathfrak{n}$-adic filtration, such that the inclusion of graded rings $\gr_f(R)\to \gr_f(S)$ extends to an isomorphism $(\gr_f(R))[X;\overline{\sigma}]\cong \gr_f(S)$ upon mapping $X$ to $\gr(x)$.

With this in mind, we note in \S \ref{subsection: lifting} that the results of \cite{wang-quantum} hold even in cases such as Example \ref{ex: nonzero graded delta}.

Finally, we note some basic structural results that seem to be absent from the literature.

\textbf{Theorem E.} Let $R$ and $S$ be as in Theorem $C$, and let $\mathfrak{n}$ be the maximal ideal of $S$. Let $f$ be either the $\mathfrak{n}$-adic filtration or the filtration of Theorem D. Then:
\begin{enumerate}[label=(\roman*)]
\item If $\gr_f(R)$ is prime, then $\gr_f(S)$ is prime, and hence $S$ is prime.
\item If $\gr_f(R)$ is a maximal order, then $\gr_f(S)$ is a maximal order, and hence $S$ is a maximal order.
\item The natural inclusion map $R\to S$ is faithfully flat.
\end{enumerate}

Part (i) is proved as Proposition \ref{propn: primality}, part (ii) as Lemma \ref{lem: max order} and part (iii) as Lemma \ref{lem: faithful flatness}.

\subsection*{Notation and conventions}

We often talk about topological rings: unless specified, the filtration on a local ring $R$ with unique maximal ideal $\mathfrak{m}$ is the $\mathfrak{m}$-adic filtration. All filtrations considered are assumed to be discrete, positive and separated. When we need to specify a filtration explicitly, we will usually denote it as a function $f: R\to \mathbb{N}\cup\{\infty\}$.

When the derivation $\delta$ is zero, we will write the skew polynomial ring $R[x;\sigma,\delta]$ (resp. the skew power series ring $R[[x;\sigma,\delta]]$) as $R[x;\sigma]$ (resp. $R[[x;\sigma]]$). Such rings are said to have \emph{(pure) automorphic type}.

\subsection*{Acknowledgements}

I am very grateful to K. A. Brown for some interesting discussions and his extensive comments on an early draft of this paper. I also gratefully acknowledge a helpful discussion with Adam Jones about soluble Iwasawa algebras; Example \ref{ex: rigid soluble IAs}(ii) and Non-example \ref{nonex: non-rigid soluble IAs} are in large part due to him.

\section{Preliminaries}

\subsection{Definitions and first results}

Write $A$ for the left $R$-module of left formal power series $$A = \left\{\sum_{i=0}^\infty r_i x^i \;\middle|\; r_i\in R\right\}$$ over $R$. Suppose the multiplication in $R$, together with the multiplication rule (\ref{eqn: multiplication}), extends to a well-defined \emph{right} $R$-module structure on $A$. Then there is a unique $R$-bimodule map $A\tensor{R} A\to A$ such that $r\otimes s\mapsto rs$ for all $r,s\in R$ and $x^i\otimes x^j\mapsto x^{i+j}$; and $A$ becomes a ring under this multiplication. In this case, we write $R[[x;\sigma,\delta]] := A$ for this ring. Note that $R[[x;\sigma,\delta]]$ contains $R$, and is a skew power series ring over $R$ as described above by definition.

When $A$ has a right $R$-module structure in this way, we will say as shorthand simply that the ring $R[[x;\sigma,\delta]]$ \emph{exists}, and otherwise that $R[[x;\sigma,\delta]]$ \emph{does not exist}.

The following computational lemma will be useful.

\begin{lem}\label{lem: multiplication by higher degrees}
Fix a positive integer $n$. Then, for $a\in R$, we have
$$x^n a = \sum_{m\in M_n} m(\sigma,\delta)(a) x^{e(m)}$$
inside the skew polynomial ring $R[x;\sigma,\delta]$ or (if it exists) the skew power series ring $R[[x;\sigma,\delta]]$. Here, $M_n$ is the set of formal (noncommutative) monomials $m = m(X,Y)$ of degree $n$ in the variables $X$ and $Y$, and $e(m)$ is the total degree of $X$ in the monomial $m$.
\end{lem}

\begin{proof}
When $n=1$, this is just the multiplication rule (\ref{eqn: multiplication}) in $S$. For $n>1$, this follows by an easy induction: note that all elements of $M_n$ are either of the form $Xm(X,Y)$ or of the form $Ym(X,Y)$ for some $m\in M_{n-1}$.
\end{proof}

\subsection{Local skew power series extensions}

There are several variant notions of ``locality" in the literature; for the purposes of this paper, we fix the following definition.

\begin{defn}\label{defn: local ring}
Let $R$ be a ring. Then $R$ is \emph{local} if it has a unique maximal two-sided ideal $\mathfrak{m}$, and $R$ is \emph{scalar local} ring if additionally $\mathfrak{m}$ is maximal among \emph{right} (and hence also \emph{left}) ideals. (It is easy to see that scalar local implies local, but that this implication is not reversible.)

An alternative formulation of this definition, which is sometimes useful, goes as follows. Let $R$ be a ring with unique maximal ideal $\mathfrak{m}$. Then $R$ is \emph{local} (resp. \emph{scalar local}) if $R/\mathfrak{m}$ is simple artinian (resp. a division ring).

We will sometimes say that \emph{$(R,\mathfrak{m})$ is a local} (resp. \emph{scalar local}) \emph{ring}, as shorthand for ``$R$ is a local (resp. scalar local) ring with unique maximal ideal $\mathfrak{m}$".

The \emph{residue ring} of a local ring $(R,\mathfrak{m})$ is defined to be the quotient $R/\mathfrak{m}$.
\end{defn}

\begin{rk}
Let $R$ be a ring with unique maximal left ideal $\mathfrak{m}$. Then it is well known that $\mathfrak{m}$ is also the unique maximal \emph{right} ideal, and is hence equal to the Jacobson radical of $R$.
\end{rk}

The applicability of our results is guaranteed by the lemma below.

\begin{defn}
Let $(R,\mathfrak{m})$ be a complete local ring. We say that $(\sigma,\delta)$ is a \emph{skew derivation} of $(R, \mathfrak{m})$ (or just $R$) if $\sigma$ is an automorphism of $R$ and $\delta$ is a (left) $\sigma$-derivation of $R$. We say that the skew derivation $(\sigma, \delta)$ is \emph{local} if $\sigma(\mathfrak{m}) \subseteq \mathfrak{m}$, $\delta(R)\subseteq \mathfrak{m}$ and $\delta(\mathfrak{m})\subseteq \mathfrak{m}^2$.
\end{defn}

\begin{lem}
\cite[Lemma 2.1]{venjakob} Let $(R, \mathfrak{m})$ be a complete local ring, and $(\sigma,\delta)$ a local skew derivation of $R$. Then the ring $R[[x;\sigma,\delta]]$ exists.\qed
\end{lem}

We do not deal with the question of existence any further in this paper; all rings written are implicitly assumed to exist.

\begin{defn}\label{defn: LSPSR over R}
Let $(R,\mathfrak{m})$ be a local ring. We will say that the ring $S$ is a \emph{local skew power series ring over $R$}, or a \emph{local skew power series extension of $R$}, if it satisfies the following properties:
\begin{enumerate}[noitemsep,label=(\roman*)]
\item $S$ is a skew power series ring over $R$ with respect to some automorphism $\sigma$ and some $\sigma$-derivation $\delta$, say $S = R[[x; \sigma, \delta]]$;
\item $\sigma(\mathfrak{m}) \subseteq \mathfrak{m}$;
\item $\delta(R) \subseteq \mathfrak{m}$ and $\delta(\mathfrak{m}) \subseteq \mathfrak{m}^2$.
\end{enumerate}
\end{defn}

\begin{rk}\label{rk: preserving local structure}
Condition (iii) implies that $\delta(\mathfrak{m}^n) \subseteq \mathfrak{m}^{n+1}$ for all $n\in\mathbb{N}$. Hence, taken together, conditions (ii) and (iii) are simply the natural requirement that the multiplication data $\sigma,\delta$ of $R$ should interact ``nicely" with the $\mathfrak{m}$-adic topology on $R$, i.e. $\sigma$ should be a homeomorphism, and $\delta$ should be topologically nilpotent.
\end{rk}

As the hypotheses above are cumbersome to repeat, we set up some shorthand notation that will remain in force for the rest of the paper.

\begin{defn}\label{defn: SPS}
Let $(R,\mathfrak{m})$ be a complete local ring. Then we will write $S\in \SPS(R)$ to mean that $S$ is a local skew power series ring $R[[x;\sigma,\delta]]$ over $R$. If $\sigma$ and $\delta$ are $A$-linear (for some subring $A\subseteq R$), we may write $S\in\SPS_A(R)$ to emphasise this fact. We may also choose to specify the maximal ideals of $R$ or $S$ explicitly, e.g. by writing $(S,\mathfrak{n})\in \SPS(R,\mathfrak{m})$ to mean that $\mathfrak{m}$ and $\mathfrak{n}$ are the maximal ideals of $R$ and $S$ respectively.
\end{defn}

We justify the name ``\emph{local} skew power series ring":

\begin{lem}\label{lem: SPS extensions are local}
Let $(R,\mathfrak{m})$ be a complete local ring, and let $S = R[[x;\sigma,\delta]]\in\SPS(R)$.
\begin{enumerate}[noitemsep,label=(\roman*)]
\item $\mathfrak{n} = \mathfrak{m} + xR$ is the unique maximal ideal of $S$, and the natural inclusion $R\subseteq S$ induces an isomorphism $S/\mathfrak{n} \cong R/\mathfrak{m}$ of residue rings.
\item $(S,\mathfrak{n})$ is local. Moreover, $(S,\mathfrak{n})$ is scalar local if and only if $(R,\mathfrak{m})$ is scalar local.
\item $S$ is $\mathfrak{n}$-adically complete.
\end{enumerate}
\end{lem}

\begin{proof}
(i) and (iii) follow from \cite[\S 2]{venjakob}. (ii) follows from (i) and the remark immediately after Definition \ref{defn: local ring}.
\end{proof}

\begin{defn}\label{defn: SPS^n}\label{defn: ILSPSR}
Let $(R,\mathfrak{m})$ be a complete local ring. Then, for each positive integer $n$, we will inductively define the notation $S\in \SPS^n(R)$ as follows: $S\in \SPS^1(R)$ if and only if $S\in \SPS(R)$; and, for $n>1$, we write $S\in \SPS^n(R)$ if and only if $S\in \SPS(T)$ for some $T\in \SPS^{n-1}(R)$.

It will be convenient to have several explicit ways of writing a ring $S\in \SPS^n(R)$. We will use either of the following equivalent forms interchangeably:
\begin{itemize}
\item $S = R[[x_1; \sigma_1,\delta_1]]\dots[[x_n;\sigma_n,\delta_n]]$,
\item $S = R[[\bm{x;\sigma,\delta}]]$ when the number $n$ is understood; and, in this case, $\bm{x} = (x_1, \dots, x_n), \bm{\sigma} = (\sigma_1, \dots, \sigma_n),$ and $\bm{\delta} = (\delta_1,\dots,\delta_n)$.
\end{itemize}

As before, we may write $S\in\SPS_A^n(R)$ if all the $\sigma_i$ and $\delta_i$ are $A$-linear.

In this way, the expression $\SPS_A^n(R)$ may be read as ``the class of $n$-fold iterated $A$-linear local skew power series extensions of $R$".

If $S\in \SPS^n(R)$, the number $n$ is called the \emph{rank} of $S$ (over $R$). (Later, we will see that this number can often be recovered from the global dimensions or Krull dimensions of $R$ and $S$.)
\end{defn}

\begin{lem}\label{lem: faithful flatness}
$S$ is a faithfully flat $R$-module.
\end{lem}

\begin{proof}
As a left $R$-module, we have $S \cong \prod_{i=0}^\infty Rx^i$, a direct product of flat $R$-modules, which is flat by \cite[Theorem 2.1]{chase}, as $R$ is noetherian. We may now conclude that $S$ is \emph{faithfully} flat as a left $R$-module using \cite[Proposition 7.2.3]{MR}: given any proper right ideal $I$ of $R$, we have $IS \subseteq \mathfrak{m}S \subseteq \mathfrak{n} \neq S$. (The same holds for $S$ as a right module by a symmetric argument.)
\end{proof}

\section{Filtered rings}

\subsection{Constructing skew derivations}

In the rare case when we wish to \emph{construct} a local skew derivation on a ring $R$, the following lemma gives us a method of doing so.

Let $k$ be a fixed local ring, and $R\in \SPS^n_k(k)$, say $R = k[[\bm{x;\sigma,\delta}]]$ with maximal ideal $\mathfrak{m}$. Let $\tau$ be a fixed automorphism of $R$ (necessarily preserving $\mathfrak{m}$).

\begin{lem}\label{lem: constructing skew ders}
Given any choice of $b_1, \dots, b_n\in\mathfrak{m}^2$, the assignment $d(x_i) = b_i$ for all $1\leq i\leq n$ extends to a unique local $k$-linear $\tau$-derivation $d$ of $R$.
\end{lem}

\begin{proof}
First, as the elements $x_1, \dots, x_n$ are $k$-linearly independent, the mapping $x_i\mapsto b_i$ (for all $1\leq i\leq n$) extends uniquely to a $k$-linear map $d: kx_1 \oplus \dots \oplus kx_n\to \mathfrak{m}^2$.

To extend $d$ to a $k$-module homomorphism $R\to R$, it will suffice to define $d(m)$ for each ordered monomial $m$ in the elements $x_1, \dots, x_n$. We do this inductively on the degree of the monomial as follows. If the monomial $m$ is of degree $1$, then $m\in \{x_1, \dots, x_n\}$, and the value of $d(m)$ is already known. Proceeding inductively, let $n>1$: if $m$ is a monomial of degree $n$, then it may be written uniquely as $x_j m'$, where $m'$ is a monomial of degree $n-1$ in $x_j, x_{j+1}, \dots, x_n$. Then we will define $d(m) := d(x_j) m' + \tau(x_j) d(m')$.

To show that this is indeed a $\tau$-derivation, we must show that the value of $d(m)$ is well-defined regardless of how it is computed. More precisely: let $m$ be an ordered monomial of total degree $s$,
$$m = x_{j_1} x_{j_2} \dots x_{j_s},$$
where $j_1 \leq j_2 \leq \dots \leq j_s$. Then, for each $1\leq r < s$, the monomial $m$ can be written as the product $m = p_r q_r$, where $p_r = x_{j_1} x_{j_2} \dots x_{j_r}$ and $q_r = x_{j_{r+1}} x_{j_{r+2}} \dots x_{j_s}$. Then we may define:
$$d_r(m) = d(p_r)q_r + \tau(p_r)d(q_r).$$
Note that $d(m)$ was defined to be $d_1(m)$. In this notation, we must show that $d_1(m) = \dots = d_{s-1}(m)$.

We will do this by induction on $s$. There is nothing to check when $s=2$. Now suppose that, for some $N$, we have established $d_1(m) = \dots = d_{t-1}(m)$ for all monomials $m$ of total degree $2\leq t\leq N$. Take a monomial $m$ of total degree $N+1$: in the above notation, we will write it as $m = x_{j_1} x_{j_2} \dots x_{j_{N+1}} = p_r q_r$ for each $1\leq r\leq N$.

Fix $1\leq r\leq N-1$. Then $m = p_r q_r = p_{r+1} q_{r+1}$, where $q_r = x_{r+1} q_{r+1}$ and $p_{r+1} = p_r x_{r+1}$, i.e.
$$m = p_r \underbrace{x_{r+1} q_{r+1}}_{q_r} = \underbrace{p_r x_{r+1}}_{p_{r+1}} q_{r+1}.$$
So, exactly as above, we may calculate
\begin{align*}
d_r(m) &= d(p_r) q_r + \tau(p_r) d(q_r)\\
&= d(p_r) x_{r+1}q_{r+1} + \tau(p_r) d(x_{r+1}q_{r+1})\\
&= d(p_r) x_{r+1}q_{r+1} + \tau(p_r) [d(x_{r+1})q_{r+1} + \tau(x_{r+1})d(q_{r+1})]\\
&= [d(p_r) x_{r+1} + \tau(p_r) d(x_{r+1})]q_{r+1} + \tau(p_r x_{r+1})d(q_{r+1})\\
&= d(p_r x_{r+1})q_{r+1} + \tau(p_r x_{r+1})d(q_{r+1})\\
&= d(p_{r+1})q_{r+1} + \tau(p_{r+1})d(q_{r+1})\\
&= d_{r+1}(m).
\end{align*}
Hence these $d_i$ are all equal, and in particular are equal to $d$. This shows that $d$ is a $\tau$-derivation as required.
\end{proof}

\subsection{Filtrations and their associated graded rings}

Let $k$ be a complete local ring and $R\in\SPS^n(k)$. Then $R$ is also a complete local ring, say with maximal ideal $\mathfrak{m}$. Studying the $\mathfrak{m}$-adic filtration and its associated graded ring is the key to understanding many basic properties of $R$.

\begin{rks}\label{rks: rank, gr, known properties}
$ $

\begin{enumerate}[noitemsep,label=(\roman*)]
\item When $k$ is a field, it is easy to see that the rank $n$ of $(R,\mathfrak{m})\in\SPS^n(k)$ over a field is uniquely defined, and can be recovered as $n = \dim_k(\mathfrak{m}/\mathfrak{m}^2)$.
\item Suppose $A$ is a central subring of $R$ and $S\in\SPS^n(R)$. Then $S\in\SPS_A^n(R)$ if and only if $A$ is a central subring of $S$ (under the natural inclusion $A\subseteq R\subseteq S$).
\item Let $(R,\mathfrak{m})$ be a complete local ring, and $(S,\mathfrak{n})\in\SPS^n(R)$. Then the restriction of the $\mathfrak{n}$-adic valuation to $R$ is the $\mathfrak{m}$-adic valuation, i.e. $R\cap \mathfrak{n}^i = \mathfrak{m}^i$. In particular, $\mathfrak{m}$ is generated in \emph{$\mathfrak{n}$-adic degree} 1.
\item Let $(R,\mathfrak{m})$ be a complete local ring, and $(S,\mathfrak{n})\in\SPS^n(R,\mathfrak{m})$. We will always write $\gr(R)$ for the graded ring associated to the $\mathfrak{m}$-adic filtration, and likewise $\gr(S)$ for the graded ring associated to the $\mathfrak{n}$-adic filtration.
\end{enumerate}
\end{rks}

We now calculate this associated graded ring inductively.

Let $R$ be a complete local ring with maximal ideal $\mathfrak{m}$, take a local skew derivation $(\sigma,\delta)$ of $R$, and form $S = R[[x;\sigma,\delta]]\in\SPS(R)$ with maximal ideal $\mathfrak{n}$. Note that $\sigma$ and $\delta$ induce linear endomorphisms of the graded ring $\gr(R)$, of degrees 0 and 1 respectively, as follows: for all $x\in \mathfrak{m}^\lambda \setminus \mathfrak{m}^{\lambda+1}$, where $\lambda\in\mathbb{N}$, we have

\begin{itemize}[noitemsep]
\item $\overline{\sigma}(x + \mathfrak{m}^{\lambda+1}) = \sigma(x) + \mathfrak{m}^{\lambda+1}$,
\item $\overline{\delta}(x + \mathfrak{m}^{\lambda+1}) = \delta(x) + \mathfrak{m}^{\lambda+2}$.
\end{itemize}

It is easy to check that $\overline{\sigma}$ is in fact a graded automorphism.

\begin{lem}\label{lem: isom of graded rings}
The unique homomorphism of graded rings $(\gr(R))[z; \overline{\sigma}, \overline{\delta}]\to \gr(S)$, extending the natural inclusion map $\gr(R)\to \gr(S)$ and sending $z$ to $x+\mathfrak{m}^2$, is an isomorphism of graded rings.
\end{lem}

\begin{proof}
Clear.
\end{proof}

\begin{rk}
It is easy to see that $\overline{\delta} = 0$ if
\begin{equation}\label{eqn: gr delta = 0}
\delta(v(x)) > v(x) + 1
\end{equation}
for all nonzero $x\in R$, where $v$ is the $\mathfrak{m}$-adic filtration: $v(x) = \lambda$ if $x\in\mathfrak{m}^{\lambda} \setminus \mathfrak{m}^{\lambda+1}$. We give an example for which $\overline\delta \neq 0$.
\end{rk}

\begin{ex}\label{ex: nonzero graded delta}
Take $R = k[[x]]$ with skew derivation $\sigma = \mathrm{id}$ and $\delta(x) = x^2$: this is indeed a skew derivation by e.g. \ref{lem: constructing skew ders}. Form $S = R[[y;\mathrm{id}, \delta]]$. It is easy to see that the graded ring has nonzero $\overline\delta$. Set $X = x+\mathfrak{n}^2$ and $Y = y+\mathfrak{n}^2$ inside $\gr(S)$: then the multiplication in $\gr(S)$ is determined by the rule
$$YX = XY + X^2,$$
i.e. $\gr(S) \cong k[X][Y; \mathrm{id}, \overline{\delta}]$, where $\overline{\delta}(X) = X^2$.
\end{ex}

Retain the above notation, and continue to write $v$ for the $\mathfrak{m}$-adic filtration on $R$.

\begin{lem}\label{lem: gr under new filtration}
The function $f: S\to \mathbb{R}\cup\{\infty\}$ defined by
$$f\left( \sum_{i=0}^\infty r_i x^i\right) = \inf\{ 2v(r_i) + i\}$$
is a Zariskian ring filtration on $S$, and the associated graded ring is
$$\gr_f(S) \cong (\gr_f(R))[z;\overline{\sigma}].$$
\end{lem}

\begin{proof}
Write $r = \sum_{i=0}^\infty r_ix^i$ and $s = \sum_{j=0}^\infty s_jx^j$, for $r_i, s_j\in R$. An application of Lemma \ref{lem: multiplication by higher degrees} shows that $f(x^i s_j) = f(\sigma^i(s_j) x^i)$ and that $f(x^i s_j - \sigma^i(s_j) x^i) > f(x^i s_j)$. The claims of the lemma now follow easily.
\end{proof}

\subsection{Lifting properties from the graded ring}\label{subsection: lifting}

Set $f_1$ to be the $\mathfrak{n}$-adic filtration on $S$, and $f_2$ to be the filtration obtained in Lemma \ref{lem: gr under new filtration}.

A few preliminary results are given for the rings $S\in\SPS^n(R)$, for $R$ a complete local ring, in \cite[Corollary 2.9]{wang-quantum}:
\begin{itemize}[noitemsep]
\item[(a)] If $\gr_{f_1}(R)$ is a domain (resp. right noetherian, resp. Auslander regular), then so is $\gr_{f_1}(S)$, and hence so is $S$.
\item[(b)] If $\gr_{f_1}(R)$ is right noetherian, then $\mathrm{r.}\Kdim(S) \leq \mathrm{r.}\Kdim(\gr_{f_1}(R)) + n$ and $\mathrm{r.}\gldim(S)$ $\leq \mathrm{r.}\gldim(\gr_{f_1}(R)) + n$.
\end{itemize}

The proofs in \cite{wang-quantum} were given under the erroneous assumption that $\overline\delta = 0$ (in the notation of Lemma \ref{lem: isom of graded rings}), but they remain true with essentially identical proofs even after removing this assumption. (We note that they also hold on replacing $f_1$ by $f_2$, of course.)

Throughout this subsection, we fix \emph{either} $f = f_1$ \emph{or} $f = f_2$, giving a filtration $f$ on $R$ and $S$. The expressions $\gr(R)$ and $\gr(S)$ will always implicitly mean $\gr_f(R)$ and $\gr_f(S)$ respectively.

We record some further important properties that lift from the graded ring.

\begin{propn}\label{propn: primality}
Let $R$ be a complete local ring, and $S\in\SPS^n(R)$. If $\gr(R)$ is prime, then $\gr(S)$ is prime, and hence $S$ is prime.
\end{propn}

\begin{proof}
If $\gr(R)$ is prime, then $\gr(S)$ is prime by Remark \ref{rks: rank, gr, known properties}(iv) along with iterated application of \cite[Theorem 1.2.9(iii)]{MR}. Hence $S$ is prime by \cite[II, Lemma 3.2.7]{LVO}.
\end{proof}

\begin{cor}\label{cor: as-gorenstein}
Let $R$ be a complete local ring, and $S\in\SPS^n(R)$. Suppose that $\gr(R)$ is Auslander regular. Then $S$ is AS-Gorenstein.
\end{cor}

\begin{proof}
$S$ is Auslander regular by \cite[2.3(iii)]{wang-quantum}, and so in particular $S$ is Auslander-Gorenstein. As $S$ is scalar local by Lemma \ref{lem: SPS extensions are local}(ii), we deduce from \cite[Lemma 4.3]{chan-wu-zhang} that $S$ is AS-Gorenstein.
\end{proof}

\begin{lem}\label{lem: max order}
Let $R$ be a complete local ring such that $\gr(R)$ is a maximal order, and take $S\in\SPS^n(R)$. Then $S$ is a maximal order.
\end{lem}

\begin{proof}
This follows from \cite[Ch. II, \S 3, Theorem 11]{LVO} alongside the above, similarly to the proof of \cite[Ch. II, \S 3, Corollary 12]{LVO}.
\end{proof}

\section{Rigid iterated extensions}

Let $(R_0,\mathfrak{m}_0)$ be a complete local ring, and let $S \in\SPS^n(R_0)$ for $n\geq 2$. That is, for each $1\leq i\leq n$, we may inductively find $R_i\in\SPS(R_{i-1})$, say $R_i = R_{i-1}[[x_i;\sigma_i,\delta_i]]$ with maximal ideal $\mathfrak{m}_i$, such that $S = R_n$. For convenience, write $\sigma = \sigma_n$, $\delta = \delta_n$.

We noted in Remark \ref{rk: preserving local structure} that we wanted each $R_i$ to interact ``nicely" with the topology of $R_{i-1}$, i.e. $(\sigma_i,\delta_i)$ should be a \emph{local} skew derivation of $(R_{i-1},\mathfrak{m}_{i-1})$ (see Definition \ref{defn: LSPSR over R}(ii--iii)). Unfortunately, naively iterating this procedure as in Definition \ref{defn: ILSPSR}, we may end up in a situation like the following. Writing the above chain of local rings explicitly,
$$(R_0,\mathfrak{m}_0)\lneq (R_1,\mathfrak{m}_1) \lneq \dots \lneq (R_{n-1},\mathfrak{m}_{n-1}) \lneq S,$$
and $(\sigma,\delta)$ is a local skew derivation of $R_{n-1}$ as desired, but we are not guaranteed that it restricts to a local skew derivation $(\sigma|_{R_i}, \delta|_{R_i})$ of $R_i$ for any $0\leq i\leq n-2$, as the following example shows.
\begin{ex}
Let $k$ be a field. Let $R_0 = k[[Y]] \lneq k[[Y,Z]] = R_1$, so that $R_0$ has maximal ideal $\mathfrak{m}_0 = (Y)$. Take $S = R[[X; \sigma, \delta]]$, where $\sigma$ is the $k$-linear automorphism of $R$ given by $\sigma(Y) = Z$, $\sigma(Z) = -Y$, and $\delta = 0$. Then $S\in\SPS_k^2(R_0)$, but $\sigma(\mathfrak{m}_0) \not\subseteq \mathfrak{m}_0$.
\end{ex}
This turns out to be a very natural stipulation to make when performing this iterative construction, and this motivates the definitions we make in this section.

\subsection{Quotients by stable ideals}
Recall the following basic definition (cf. \cite[3.13]{letzter-noeth-skew}).

\begin{defn}
Let $R$ be a ring and $I$ an ideal. If $(\sigma,\delta)$ is a skew derivation of $R$, then $I$ is said to be a \emph{$(\sigma,\delta)$-ideal} if $\sigma(I)\subseteq I$ and $\delta(I)\subseteq I$.
\end{defn}

This is a useful class of ideals because of results such as the following lemma:

\begin{lem}\label{lem: quotient by ideals}
Let $R$ be a complete local ring, and $S = R[[x;\sigma,\delta]]\in\SPS(R)$. Suppose that $I$ is a $(\sigma,\delta)$-ideal of $R$. Then $IS = SI$ is a two-sided ideal, and $S/IS \cong (R/I)[[x;\sigma,\delta]]\in\SPS(R/I)$.
\end{lem}

\begin{proof}
Note that \cite[Setup 3.1(4)]{letzter-noeth-skew} is satisfied by Remark \ref{rks: rank, gr, known properties}(iv). Hence we may apply \cite[3.13(ii) and Lemma 3.14(iv)]{letzter-noeth-skew}.
\end{proof}

We slightly extend this notion as follows.

\begin{defn}
Let $R$ be a ring and $I$ an ideal. Let $\bm{\sigma} = (\sigma_1, \dots, \sigma_r)$ and $\bm{\delta} = (\delta_1, \dots, \delta_r)$, where each $(\sigma_i, \delta_i)$ is a skew derivation of $R$. We will say that $I$ is a \emph{$(\bm{\sigma,\delta})$-ideal} if it is a $(\sigma_i,\delta_i)$-ideal for each $1\leq i\leq r$.
\end{defn}

\begin{propn}\label{propn: quotient by ideals}
Let $R$ be a complete local ring, and $S = R[[\bm{x;\sigma,\delta}]]\in\SPS^n(R)$. Suppose that $I$ is a $(\bm{\sigma,\delta})$-ideal of $R$. Then $IS = SI$ is a two-sided ideal, and $S/IS \cong (R/I)[[\bm{x;\sigma,\delta}]]\in\SPS^n(R/I)$.
\end{propn}

\begin{proof}
Both claims follow from recursive application of Lemma \ref{lem: quotient by ideals}.
\end{proof}

\subsection{Skew power series presentations}

\begin{defn}\label{defn: presn}
Fix a complete local ring $(R_0, \mathfrak{m}_0)$ and a ring $R\in \SPS^n(R_0)$. A \emph{presentation} for $R$ (over $R_0$) is a sequence of rings
\begin{equation}\label{eqn: presn of R}
R_0 \lneq R_1 \lneq \dots \lneq R_{\ell-1} \lneq R_\ell = R,
\end{equation}
where $R_i\in\SPS^{d_i}(R_{i-1})$ for each $1\leq i\leq \ell$, and $d_1 + \dots + d_\ell = n$. The number $\ell$ is called the \emph{length} of the presentation. 
\end{defn}

\begin{rk}\label{rk: upper bound on presn length}
Note that, by definition, each $d_i\geq 1$, so the length $\ell$ of any presentation (\ref{eqn: presn of R}) must be bounded above by the rank $n$ of $R$ over $R_0$.
\end{rk}

\begin{defn}\label{defn: stability}
With $R\in\SPS^n(R_0)$ as above, denote $R = R_0[[\bm{x;\sigma,\delta}]]$. We will say that the presentation (\ref{eqn: presn of R}) is \emph{stable} if the maximal ideal of $R_i$ is a $(\sigma_j,\delta_j)$-ideal for each pair $(i,j)$ satisfying $0\leq i < j\leq \ell$.
\end{defn}

\begin{defn}
We say that $R\in\SPS^n(R_0)$ is \emph{rigid} if it admits a stable presentation of length $n$, i.e. a saturated presentation which is also stable, and we will denote this by $R\in\RSPS^n(R_0)$. (As before, we will write $\RSPS_A^n(R_0) = \SPS_A^n(R_0)\cap \RSPS^n(R_0)$.)
\end{defn}

\begin{propn}\label{propn: quotient extensions}
Let $R_0$ be a complete local ring with residue ring $k$, and take some $R\in\RSPS^n(R_0)$, say with stable presentation $(R_0,\mathfrak{m}_0) \lneq (R_1,\mathfrak{m}_1) \lneq \dots \lneq (R_n,\mathfrak{m}_n) = (R,\mathfrak{m})$. Then $R/\mathfrak{m}_iR\in \RSPS^{n-i}(k)$.
\end{propn}

\begin{proof}
When $i=0$, we may apply Proposition \ref{propn: quotient by ideals}, and the result is immediate.

When $i>0$, set $R'_0 = R_i$, so that $R'_0$ is a complete local ring with residue ring $k$ by Lemma \ref{lem: SPS extensions are local}. We may view $R$ as an element of $\RSPS^{n-i}(R'_0)$, with stable presentation $(R'_0,\mathfrak{m}'_0) \lneq (R'_1 ,\mathfrak{m}'_1)\lneq \dots \lneq (R'_{n-i},\mathfrak{m}'_{n-i}) = R,$
where $R'_t = R_{t+i}$ and $\mathfrak{m}'_t = \mathfrak{m}_{t+i}$. Now the result follows from the previous case.
\end{proof}

\subsection{Examples}\label{subsection: examples}

It turns out that many interesting iterated local skew power series extensions arising in nature are indeed rigid. Examples \ref{ex: q-commutatives}--\ref{ex: rigid soluble IAs} below were important motivating examples for this paper, and will be referred to later.

\begin{ex}\label{ex: q-commutatives}
\textbf{$q$-commutative power series rings.} Let $k$ be a field, and $q\in M_n(k^\times)$ a \emph{multiplicatively antisymmetric} matrix (i.e. $q_{ij} q_{ji} = 1$ for all $i$ and $j$). Then the \emph{$q$-commutative power series ring} is the ring
$$R = k_q[[x_1, \dots, x_n]],$$
defined as follows: $R \cong k[[x_1, \dots, x_n]]$ as topological $k$-modules; and the multiplication is given by the $n^2$ relations $x_i x_j = q_{ij} x_j x_i$. It is easy to see that $R\in\RSPS_k^n(k)$.
\end{ex}

For the following two examples, we assume familiarity with the notion of \emph{uniform groups} \cite{DDMS} and \emph{completed group rings} (also known as \emph{Iwasawa algebras}) \cite{lazard}.

\begin{ex}\label{ex: nilpotent IAs}
\textbf{Nilpotent Iwasawa algebras.} Let $G$ be a nilpotent uniform group. $G$ admits a series
$$\{1\} = G_0 \leq G_1 \leq \dots \leq G_n = G$$
of closed subgroups of $G$ such that $G_{i-1}$ is normal in $G$ and $G_i/G_{i-1} \cong \mathbb{Z}_p$ for all $1\leq i\leq n$: an appropriate refinement of the isolated lower central series, defined in \cite[Definition 3.7]{woods-struct-of-G}, will suffice. Hence $G$ is \emph{supersoluble}, as defined below. Now it will follow from Example \ref{ex: rigid soluble IAs} below that its appropriate Iwasawa algebras are rigid.
\end{ex}

\begin{ex}\label{ex: rigid soluble IAs}
\textbf{Supersoluble Iwasawa algebras.} Let $G$ be a \emph{supersoluble} uniform group, i.e. there exists a sequence
$$\{1\} = G_0 \leq G_1 \leq \dots \leq G_n = G$$
of closed subgroups of $G$ such that $G_{i-1}$ is normal in $G$ and $G_i/G_{i-1} \cong \mathbb{Z}_p$ for all $1\leq i\leq n$. Fix such a sequence. Let $k$ be \emph{either} the ring of integers of a finite extension of $\mathbb{Q}_p$ \emph{or} a finite field extension of $\mathbb{F}_p$. Now, as in \cite[Example 2.3]{venjakob}, $kG\in\SPS^n_k(k)$ with presentation
$$k = kG_0 \leq kG_1 \leq \dots \leq kG_n = kG,$$
and the condition that $G_{i-1}$ should be normal in $G$ ensures that this presentation is stable. Hence $kG\in\RSPS^n_k(k)$.

We demonstrate the existence of some supersoluble, non-nilpotent uniform groups of small rank below. In both cases, uniformity can be easily checked using \cite[Theorem 4.5]{DDMS}.
\begin{enumerate}[label=(\roman*)]
\item Let $\Gamma_1$ and $\Gamma_2$ be copies of $\mathbb{Z}_p$, and take the continuous group homomorphism $\rho: \Gamma_2\to \Aut(\Gamma_1)\cong \mathbb{Z}_p^\times$ sending a generator of $\Gamma_2$ to $1+p\in\mathbb{Z}_p^\times$. Form the semidirect product $G = \Gamma_1\rtimes \Gamma_2$. Then $G$ is a supersoluble but non-nilpotent uniform group of dimension 2. (cf. \cite[Example 2.2]{venjakob})
\item Take $p$ to be a prime congruent to $1$ mod $4$, so that $i := \sqrt{-1}\in k$. Let $A = \overline{\langle y, z\rangle} \cong \mathbb{Z}_p^2$ and $B = \overline{\langle x\rangle} \cong \mathbb{Z}_p$, and fix the left action of $B$ on $A$, say $\rho: B\to \Aut(A)$, defined by $\rho(x)(y) = yz^p$ and $\rho(x)(z) = zy^{-p}$. Then $G = B\rtimes A$ is easily checked to be a soluble, non-nilpotent uniform group of dimension $3$, and the chain of normal subgroups
$$1 \leq \overline{\langle yz^i\rangle} \leq B \leq A$$
shows that $G$ is in fact supersoluble. (Compare Non-example \ref{nonex: non-rigid soluble IAs}.)
\end{enumerate}
\end{ex}

The next two examples are not crucial for the current paper, but we include them to illustrate the wide applicability of Theorems B and D.

\begin{ex}\label{ex: horton's algs}
\textbf{Completed quantised $k$-algebras.} Let $k$ be a field, $\Gamma = (\gamma_{ij})\in M_n(k^\times)$ a multiplicatively antisymmetric matrix, and $P = (p_1, \dots, p_n)\in (k^\times)^n$, $Q = (q_1, \dots, q_n)\in (k^\times)^n$ two vectors with $p_i\neq q_i$ for all $1\leq i\leq n$. Horton's algebra $R = K_{n,\Gamma}^{P,Q}(k)$, a simultaneous generalisation of quantum symplectic space and quantum Euclidean $2n$-space, can be written as an iterated skew polynomial ring,
$$k[x_1][y_1; \tau_1][x_2; \sigma_2][y_2; \tau_2, \delta_2]\dots [x_n;\sigma_n][y_n; \tau_n, \delta_n],$$
where the $\sigma_i$ (for $2\leq i\leq n$) and $\tau_i$ (for $1\leq i\leq n$) are $k$-linear automorphisms, and each $\delta_i$ (for $2\lneq i\lneq n$) is a $k$-linear $\tau_i$-derivation. In \cite[\S 3.2]{wang-quantum}, it is proved that the $I$-adic completion $\hat{R}$ of $R$, where $I = (x_1, y_1, x_2, y_2, \dots, x_n, y_n)$, is an iterated skew power series extension of $k$:
\begin{equation}\label{eqn: completed horton}
\hat{R} = k[[x_1]][[y_1; \tau_1]][[x_2; \sigma_2]][[y_2; \tau_2, \delta_2]]\dots [[x_n;\sigma_n]][[y_n; \tau_n, \delta_n]]\in\SPS^{2n}_k(k).
\end{equation}
We do not spell out the relations in full: see \cite[Proposition 3.5]{horton} for details. It is only necessary, for our purposes, to know the following:
\begin{itemize}[noitemsep]
\item $\sigma_i(x_j), \tau_i(x_j), \delta_i(x_j)$ are scalar multiples of $x_j$ for all $j < i$;
\item $\sigma_i(y_j), \tau_i(y_j), \delta_i(y_j)$ are scalar multiples of $y_j$ for all $j < i$;
\item $\tau_i(x_i)$ is a scalar multiple of $x_i$ for all $i$;
\item $\delta_i(x_i)$ is a $k$-linear combination of the elements $y_lx_l$ for all $l<i$.
\end{itemize}
It now follows by an easy calculation that the saturated presentation associated to (\ref{eqn: completed horton}) is stable, and hence that $\hat{R}\in\RSPS^{2n}_k(k)$.
\end{ex}

\begin{ex}\label{ex: quantum matrices}
\textbf{Completed quantum matrix algebras.} Let $k$ be a field, $\lambda\in k^\times$ a scalar, and $\mathbf{p} = (p_{ij}) \in M_n(k^\times)$ a multiplicatively antisymmetric matrix (i.e. $p_{ij} p_{ji} = 1$ for all $i, j$). Then the \emph{multiparameter quantum $n\times n$ matrix algebra} $R = \mathcal{O}_{\lambda, \mathbf{p}}(M_n(k))$ can be defined (see e.g. \cite[Definition I.2.2]{brown-goodearl}) as a skew polynomial ring in $n^2$ variables labelled $X_{i,j}$ for each $1\leq i, j\leq n$, in which $k$ is central. Again, we do not spell out the relations in full, but we note:
\begin{itemize}[noitemsep]
\item $X_{l,m} X_{i,j}$ is a linear combination of $X_{i,j} X_{l,m}$ and $X_{i,m} X_{l,j}$ when $l>i$ and $m>j$;
\item $X_{l,m} X_{i,j}$ is a scalar multiple of $X_{i,j} X_{l,m}$ whenever either $l\leq i$ or $j\leq m$.
\end{itemize}
In \cite[\S 3.2]{wang-quantum} it is proved that the $I$-adic completion $\hat{R}$ of $R$, where $I$ is the ideal generated by the $n^2$ variables $X_{i,j}$ for $1\leq i,j\leq n$, is an iterated skew power series extension of $k$ satisfying the same relations: $\hat{R}\in\SPS^{n^2}_k(k)$. But the ``obvious" saturated presentations, e.g. those associated to
$$k[[X_{1,1}]][[X_{1,2}]][[X_{1,3}]]\dots[[X_{n,n}]]$$
(where we have omitted the skew derivations for readability) are usually \emph{not} stable. (For a counterexample, take $R$ to be the usual quantum $2\times 2$ matrix algebra $\mathcal{O}_q(M_2(k))$, given by $n=2$, $\mathbf{p} = \left(\begin{smallmatrix} 1&q\\q^{-1}&1\end{smallmatrix}\right)$, and $\lambda = q^{-2}$ for any $q\in k^\times$. See \cite[Definition I.1.7]{brown-goodearl}) for the relations in this case.)

We fix this by adjoining the variables in the following order:
\begin{itemize}[noitemsep]
\item at the \emph{0th stage}, adjoin the ``antidiagonal" elements $X_{i,j}$ satisfying $|i+j - (n+1)| = 0$, in any order;
\item at the \emph{1st stage}, adjoin those $X_{i,j}$ satisfying $|i+j - (n+1)| = 1$;
\item at the \emph{2nd stage}, adjoin those $X_{i,j}$ satisfying $|i+j-(n+1)| = 2$;
\end{itemize}
and so on, until finally all variables have been adjoined at the end of the $(n-1)$th stage.

Diagrammatically:

$$\underset{\text{0th stage}}{
\begin{pmatrix}
&&&&\iddots\\
&&&\circledast\\
&&\circledast\\
&\circledast\\
\circledast
\end{pmatrix}
}
\rightsquigarrow
\underset{\text{1st stage}}
{
\begin{pmatrix}
\vspace{-7pt}&&&\iddots&\iddots\\
&&\circledast&*&\iddots\\
&\circledast&*&\circledast\\
\circledast&*&\circledast\\
*&\circledast
\end{pmatrix}
}
\rightsquigarrow
\underset{\text{2nd stage}}{
\begin{pmatrix}
\vspace{-7pt}&&\iddots&\iddots&\iddots\\
\vspace{-7pt}&\circledast&*&*&\iddots\\
\circledast&*&*&*&\iddots\\
*&*&*&\circledast\\
*&*&\circledast
\end{pmatrix}
}
\rightsquigarrow \dots
$$

It is easy to verify that such a presentation of $\hat{R}$ is stable. The only case in which there is anything to prove is when $l>i$ and $m>j$: multiplying $X_{l,m}$ by $X_{i,j}$ results in a term involving the variables $X_{i,m}$ and $X_{l,j}$, and we must check that each of the variables $X_{i,m}$ and $X_{l,j}$ has been adjoined \emph{before} we adjoin both of $X_{l,m}$ and $X_{i,j}$.

More precisely, let $M = |\max\{i+j, l+m\} - (n+1)|$ be the first stage after which both $X_{l,m}$ and $X_{i,j}$ have been adjoined, and likewise let $N = |\max\{i+m, l+j\} - (n+1)|$ be the first stage after which both $X_{i,m}$ and $X_{l,j}$ have been adjoined. Then it is easy to see that the inequalities $l>i$ and $m>j$ imply
$$i+j < i+m < l+m,$$
$$i+j < l+j < l+m,$$
and hence $N < M$.

The existence of such a stable presentation shows that $\hat{R}\in\RSPS^{n^2}_k(k)$.
\end{ex}

\subsection{Non-examples}\label{subsection: non-exs}

Showing that an extension is \emph{not} rigid appears to involve lots of tedious calculation, but we give two examples which seem of interest to the theory.

\begin{nonex}\label{nonex: pure aut}
\textbf{An extension of pure automorphic type.} Let $k$ be a field such that $\sqrt{-1}\not\in k$, and let $S = k[[Y,Z]][[X;\sigma]]$, where $\sigma$ is a $k$-linear automorphism of the commutative power series ring $k[[Y,Z]]$ defined by $\sigma(Y) = Z$, $\sigma(Z) = -Y$: that is, $ZX = XY$ and $YX = -XZ$ inside $S$.

By construction, $S\in\SPS_k^3(k)$.

Suppose that $S$ is rigid, and denote its maximal ideal by $\mathfrak{n}$. Then there must exist a stable presentation
$$k \lneq (R_1,\mathfrak{m}_1) \lneq (R_2,\mathfrak{m}_2) \lneq (S,\mathfrak{n})$$
for $S$. We will show that no such stable presentation can exist. Write $I = \mathfrak{m}_1S$ and $J = \mathfrak{m}_2S$: then, by Lemma \ref{lem: SPS extensions are local}(i), $I = sS$ and $J = sS + tS$ with $s, t\in\mathfrak{n}$, and $s,t\not\in\mathfrak{n}^2$ by Remark \ref{rks: rank, gr, known properties}(iii).

For computation purposes, we will write
\begin{itemize}[noitemsep]
\item $s = aX + bY + cZ + \varepsilon$,
\item $t = \alpha X+\beta Y+\gamma Z+\varepsilon'$,
\end{itemize}
where $\varepsilon,\varepsilon'\in\mathfrak{n}^2$, and $a, b, c, \alpha, \beta, \gamma\in k$ are constants, which must satisfy $(a, b, c) \neq (0,0,0) \neq (\alpha,\beta,\gamma)$ by Remark \ref{rks: rank, gr, known properties}(iii).

As $S$ is rigid, $I = \mathfrak{m}_1S$ must in fact be a two-sided ideal. This places some restrictions on possible choices for $s$, which we now compute.

Henceforth, we work in $S/\mathfrak{n}^3$ for ease of computation. We have
$$Xs \equiv aX^2 + bXY + cXZ,\qquad sX \equiv aX^2 + cXY - bXZ\mod\mathfrak{n}^3,$$
and so
$$Xs - sX = (b-c)XY + (b+c)XZ\mod\mathfrak{n}^3.$$
But $Xs-sX\in I = sS$,  so we must have $(b-c)XY + (c+b)XZ \equiv s\alpha \mod\mathfrak{n}^3$ for some $\alpha\in S$. As the left-hand side belongs to $\mathfrak{n}^2$, we must have $\alpha\in\mathfrak{n}$, and so $\varepsilon\alpha\in\mathfrak{n}^3$. Writing therefore $\alpha \equiv d_XX + d_YY + d_ZZ\mod\mathfrak{n}^2$ for some $d_X, d_Y, d_Z\in k$, we see that
$$(b-c)XY + (b+c)XZ \equiv (aX + bY + cZ)(d_XX + d_YY + d_ZZ)\mod\mathfrak{n}^3.$$
Multiplying out the right-hand side:
\begin{align*}
(b-c)XY + (b+c)XZ &\equiv ad_XX^2 + (ad_Y+cd_X)XY + (ad_Z-bd_X)XZ\\
&\qquad + bd_YY^2 + (bd_Z+cd_Y)YZ + cd_ZZ^2  &\mod\mathfrak{n}^3.
\end{align*}
Now, equating the coefficients of each monomial on both sides, some tedious case-checking shows that the only solution to this congruence is $b=c=d_X=d_Y=d_Z=0$.

Hence we have $s = aX + \varepsilon$, so that $a\neq 0$.

Now, as $S$ is rigid, $J = \mathfrak{m}_2S$ must also be a two-sided ideal, and so we calculate the restrictions that this places on $t$.

As $R_1\in\SPS_k(k)$, it must be a \emph{commutative} power series ring over $k$, so that $R_1 \cong k[[s]]$. Then $R_2 = k[[s]][[t;\tau,\delta]]$ for some local skew derivation ($\tau,\delta$) of $k[[s]]$, and an easy calculation shows that there must be a unit $\eta\in k[[s]]^\times$ such that $\tau(s) = \eta s$, and there must be an element $\theta\in k[[s]]$ such that $\delta(s) = \theta s^2$. (Note that $\eta\not\equiv 0 \mod\mathfrak{n}^2$.) Hence we have $\tau(s) \equiv \eta s \equiv \eta aX \mod \mathfrak{n}^2$ and $\delta(s)\equiv \theta s^2 \equiv a^2\theta X^2\mod\mathfrak{n}^3$, and so
$$(\alpha X + \beta Y + \gamma Z)\eta aX \equiv \eta aX(\alpha X + \beta Y + \gamma Z) + a^2\theta X^2 \mod \mathfrak{n}^3.$$

Multiplying out again:
$$a^2\theta X^2 + \eta a (\beta-\gamma)XY + \eta a(\beta+\gamma)XZ \equiv 0 \mod \mathfrak{n}^3,$$
from which we see immediately that $\beta=\gamma=\theta=0$, and hence $t = \alpha X + \varepsilon'$.

This implies that the images of $s$ and $t$ in $\mathfrak{n}/\mathfrak{n}^2$, and hence also in $\mathfrak{m}_2/\mathfrak{m}_2^2$, are linearly dependent, and so $\dim_k(\mathfrak{m}_2/\mathfrak{m}_2^2) \leq 1$. But $R_2\in\SPS^2_k(k)$ by construction, and Remark \ref{rks: rank, gr, known properties}(i) tells us that we should have $\dim_k(\mathfrak{m}_2/\mathfrak{m}_2^2) = 2$. This is a contradiction.
\end{nonex}

\begin{nonex}\label{nonex: non-rigid soluble IAs}
\textbf{A soluble Iwasawa algebra.}

This example is similar to the previous in many ways. Compare also Example \ref{ex: rigid soluble IAs}(ii).

Fix a prime $p$ congruent to $3$ mod $4$, so that $\sqrt{-1}\not\in\mathbb{F}_p$. Let $A = \overline{\langle y, z\rangle} \cong \mathbb{Z}_p^2$ and $B = \overline{\langle x\rangle} \cong \mathbb{Z}_p$, and form the semidirect product $G = B\rtimes A$ as in Example \ref{ex: rigid soluble IAs}(ii). Construct its $\mathbb{F}_p$-Iwasawa algebra $S = \mathbb{F}_pG$.

Writing $X = x-1, Y = y-1, Z = z-1$, we may easily calculate (see e.g. \cite[Example 2.3]{venjakob}) that
$$S = \mathbb{F}_p[[Y, Z]][[X; \sigma,\delta]],$$
where $\sigma(Y) = Y + (1+Y)Z^p$, $\sigma(Z) = Z + (-Y+Y^2-Y^3+\dots)^p(1+Z)$, and $\delta = \sigma-\mathrm{id}$. Clearly, $S\in\SPS^3_{\mathbb{F}_p}(\mathbb{F}_p)$.

Suppose that $S$ is rigid, and denote its maximal ideal by $\mathfrak{n}$. Then there must exist a stable presentation
$$\mathbb{F}_p \lneq (R_1,\mathfrak{m}_1) \lneq (R_2,\mathfrak{m}_2) \lneq S$$
for $S$: in particular, by Proposition \ref{propn: quotient extensions}, $S$ must admit a quotient $S/\mathfrak{m}_1S\in\RSPS^2_{\mathbb{F}_p}(\mathbb{F}_p)$. We will show that this leads to a contradiction.

As in Non-example \ref{nonex: pure aut}, take $I = \mathfrak{m}_1 S = sS$, and write $s = aX + bY + cZ + \varepsilon$, where $\varepsilon\in\mathfrak{n}^2$, and $a, b, c\in\mathbb{F}_p$. By rigidity of $S$, $I$ must be a two-sided ideal; this time, we calculate restrictions on $s$ by working in $S/\mathfrak{n}^{p+1}$. We have
$$\sigma(Y) \equiv Y + Z^p , \qquad \sigma(Z)\equiv Z -Y^p, \mod \mathfrak{n}^{p+1},$$
and so we must have
$$Xs - sX \equiv bZ^p - cY^p \mod \mathfrak{n}^{p+1}.$$
But $Xs-sX\in I = sS$, so we must have $bZ^p - cY^p \equiv s\alpha \mod \mathfrak{n}^{p+1}$ for some $\alpha\in S$. As the left-hand side belongs to $\mathfrak{n}^p$, we must have $\alpha\in\mathfrak{n}^{p-1}$, and so $\varepsilon\alpha\in\mathfrak{n}^{p+1}$. So this equation becomes
$$bZ^p - cY^p \equiv (aX + bY + cZ)\left(\sum_{\gamma} d_\gamma X^{\gamma_1} Y^{\gamma_2} Z^{\gamma_3}\right) \mod \mathfrak{n}^{p+1},$$
for some choices of $d_\gamma\in \mathbb{F}_p$, where this sum ranges over all $\gamma = (\gamma_1, \gamma_2, \gamma_3)$ with $p-1 \leq \gamma_1+\gamma_2+\gamma_3 \leq p$.

Suppose that, for our given $a,b,c$, we have a solution $\{d_\gamma\}$ to this congruence. Then, multiplying out the right-hand side, and writing $\mathbf{e}_1 = (1,0,0), \mathbf{e}_2 = (0,1,0), \mathbf{e}_3 = (0,0,1)$:
$$bZ^p - cY^p \equiv \sum_{\gamma'} \left(ad_{\gamma' - \mathbf{e}_1} + bd_{\gamma' - \mathbf{e}_2} + cd_{\gamma' - \mathbf{e}_3}\right) X^{\gamma'_1} Y^{\gamma'_2} Z^{\gamma'_3} \mod \mathfrak{n}^{p+1},$$
where for convenience we set $d_\gamma = 0$ if any of the $\gamma_i$ is equal to $-1$. We may eliminate any term in the sum of total degree not equal to $p$, as all nonzero monomials appearing on the left hand side have degree $p$, so this sum may be taken to range over all $\gamma' = (\gamma'_1, \gamma'_2, \gamma'_3)$ with $\gamma'_1+\gamma'_2+\gamma'_3 = p$; furthermore, as there are no terms in $X^{\gamma'_1} Y^{\gamma'_2} Z^{\gamma'_3}$ on the left hand side except those with $\gamma_1 = 0$, we may also eliminate all monomials of nonzero degree in $X$ on the right. We can now rewrite this congruence as
$$bZ^p - cY^p \equiv \sum_{i=0}^p \left(bd_{(0,i-1,p-i)} + cd_{(0,i,p-i-1)}\right) Y^{i} Z^{p-i} \mod \mathfrak{n}^{p+1}.$$
It is easy to see that, if $b\neq 0$, then $c\neq 0$, and vice-versa. We assume for contradiction that $b\neq 0\neq c$: we will show that, in this case, the above congruence cannot hold for any choice of $\{d_\gamma\}$.

Indeed, equating monomial coefficients on the left and right hand sides:
\begin{align*}
b &= cd_{(0,0,p-1)} && (i=0)\\
0 &= bd_{(0,m-1,p-m)} + cd_{(0,m,p-m-1)} && (i=m : 1\leq m\leq p-1)\\
-c &= bd_{(0,p-1,0)} && (i=p).
\end{align*}
On the one hand, multiplying the equations labelled $(i=0)$ and $(i=p)$ together, we get
$$-bc = bcd_{(0,0,p-1)} d_{(0,p-1,0)},$$
i.e. $d_{(0,0,p-1)} d_{(0,p-1,0)} = -1$. On the other hand, multiplying the equation labelled $(i=m)$ by $b^{p-m} c^m$ and rearranging for each $1\leq m\leq p-1$, we get
$$b^{p-m+1} c^m d_{(0,m-1,p-m)} = -b^{p-m}c^{m+1}d_{(0,m,p-m-1)};$$
substituting each one into the next, we eventually get
$$b^pcd_{(0,0,p-1)} = bc^pd_{(0,p-1,0)}$$
as $p$ is odd; and since $b,c\in\mathbb{F}_p$, we have $b^p = b$ and $c^p = c$. This tells us that $d_{(0,0,p-1)} = d_{(0,p-1,0)}$, and denoting this common value by $d$, we have shown that we must have $d^2 = -1\in\mathbb{F}_p$, which is a contradiction.

Hence we have shown that $b=c=0$, and so $s = aX + \varepsilon$. But now
$$sY - Ys \equiv aZ^p\mod\mathfrak{n}^{p+1},$$
and a very similar (but easier) calculation shows that we must have $a=0$. Hence $s = \varepsilon\in\mathfrak{n}^2$: that is, $\mathfrak{m}_1$ is generated in $\mathfrak{n}$-adic degree $\geq 2$. This contradicts Remark \ref{rks: rank, gr, known properties}(iii).
\end{nonex}

\section{Dimension theory}

Many of the results in this section can be slightly extended; but we will not always strive for full generality, and often work over a field or a division ring for simplicity.

\subsection{Krull dimension}

Let $k$ be a division ring, and $(R,\mathfrak{m})\in \SPS^n(k)$. Remark \ref{rks: rank, gr, known properties}(v)(c) implies that $\Kdim(R) \leq n$: in this subsection, we show that this is always an equality.

For any ring $R$, write $\mathcal{I}_r(R)$ for the lattice of right ideals of $R$.

The following is the result corresponding to \cite[3.14(ii)]{letzter-noeth-skew} in the case when $I$ is \emph{not} a $(\sigma,\delta)$-ideal.

\begin{propn}\label{propn: poset map theta}
Let $R$ be a complete local ring, $S = R[[x;\sigma,\delta]]\in\SPS(R)$, and let $I$ be a right ideal of $R$. Set $I[[x;\sigma,\delta]] := \{\sum a_i x^i: a_i \in I\}$. Then $I[[x;\sigma,\delta]]$ is a right ideal of $S$. Moreover, the map $\theta: \mathcal{I}_r(R) \to \mathcal{I}_r(S)$ sending $I$ to $I[[x;\sigma,\delta]]$ is a strictly increasing poset map.
\end{propn}

\begin{proof}
$I[[x;\sigma,\delta]]$ is clearly an additive subgroup of $S$.

Take $r\in R$ and $a = \sum_{i=0}^\infty a_i x^i\in I[[x;\sigma,\delta]]$. Then we may evaluate $ar$ inside $S$:
\begin{align*}
ar &= \sum_{i=0}^\infty a_i x^i r\\
&= \sum_{i=0}^\infty \sum_{m\in M_i} a_i m(\sigma,\delta)(r) x^{e(m)}
\end{align*}
in the notation of Lemma \ref{lem: multiplication by higher degrees}. Now, $m(\sigma,\delta)(r)\in R$, so $a_i m(\sigma,\delta)(r)\in I$, and hence $ar\in I[[x;\sigma,\delta]]$.

We need to check that this gives a well-defined right $S$-action on $I[[x;\sigma,\delta]]$ -- in other words, that the right actions of the elements $xr$ and $\sigma(r)x + \delta(r)$ agree for all $r\in R$. But this is already true \emph{a fortiori}, as the right action of $S$ on $I[[x;\sigma,\delta]]$ is just induced by multiplication inside the ring $S$.

Finally, it is clear that, if $I_1\leq I_2$, then $\theta(I_1)\leq \theta(I_2)$; and, if $J = \theta(I)$, then we may recover $I$ as $J/Jx$. This shows that $\theta$ is a strict map of posets.
\end{proof}

\begin{lem}\label{lem: induced simple module is k[[x]]}
Let $(R,\mathfrak{m})$ be a complete local ring with residue ring $k$. Let $Y\subseteq X$ be adjacent right ideals of $R$: then, as right $S$-modules, $\theta(X)/\theta(Y)$ is isomorphic to $S/\mathfrak{m}S \cong k[[\overline{x}; \overline{\sigma}]]$, a skew power series ring of automorphic type.
\end{lem}

\begin{proof}
Let $\theta$ continue to denote the map $\mathcal{I}_r(R)\to \mathcal{I}_r(S)$ defined in Proposition \ref{propn: poset map theta}.

$Y\subset X$ are adjacent if and only if $X/Y$ is a simple right module. In particular, the annihilator $\mathrm{ann}(X/Y)_R$ must be the unique maximal right ideal $\mathfrak{m}$ of $R$, and so we must have $X/Y\cong R/\mathfrak{m} \cong k$.

The proposed isomorphism is obvious on the level of abelian groups, and it is easy to see that right multiplication by $x\in S$ is the same as right multiplication by $\overline{x}\in k[[\overline{x};\overline{\sigma}]]$. It remains to check the $R$-action. Recall, from Lemma \ref{lem: multiplication by higher degrees}, that
$$\left(\sum_{i=0}^\infty a_i x^i\right) r = \sum_{i=0}^\infty \sum_{m\in M_i} a_i m(\sigma,\delta)(r) x^{e(m)}$$
for all $a_i\in X, r\in R$. But, for a given $m\in M_i$, if $m(\sigma,\delta)$ contains an instance of $\delta$ (i.e. if $e(m) < i$), then $m(\sigma,\delta)(r) \in \mathfrak{m}$, and hence $a_i m(\sigma,\delta)(r) \in X\mathfrak{m} \subseteq Y$. This implies that
$$\left(\sum_{i=0}^\infty \overline{a_i x^i}\right) \overline{r} = \sum_{i=0}^\infty \overline{a_i \sigma^i(r) x^i}$$
inside $X/Y$, as required.
\end{proof}

\begin{thm}\label{thm: Kdim}
Let $k$ be a division ring, and $S\in\SPS^n(k)$. Then $\Kdim(S) = \Kdim(\gr(S)) = n$.
\end{thm}

\begin{proof}
The proof of this theorem closely follows some of the methods of \cite{ardakov-Kdim}. We will calculate the \emph{right} Krull dimension of $S$, but the calculation of the left Krull dimension is identical. Let $\theta$ continue to denote the map $\mathcal{I}_r(R)\to \mathcal{I}_r(S)$ of Proposition \ref{propn: poset map theta}.

When $n = 0$, we have $S = \gr(S) = k$, and there is nothing to prove. We proceed by induction on the rank of $S$.

Let $R\in\SPS^{n-1}(k)$, with $S = R[[x;\sigma,\delta]]\in\SPS(R)$. By the inductive hypothesis, we know that $$\Kdim(R) = \Kdim(\gr(R)) = n-1.$$ We also know, by \cite[Corollary 2.9(ii)]{wang-quantum}, that $\gr(S) \cong \gr(R)[\overline{x};\overline{\sigma}]$, a skew polynomial ring of automorphic type, and so $\Kdim(\gr(S)) = \Kdim(\gr(R)) + 1 = n$ by \cite[Proposition 6.5.4(i)]{MR}.

Now, given arbitrary adjacent right ideals $Y\subseteq X$ of $R$, we know already from Lemma \ref{lem: induced simple module is k[[x]]} that $\theta(X)/\theta(Y) \cong k[[\overline{x};\overline{\sigma}]]$ as right $S$-modules: in particular, $\theta(X)/\theta(Y)$ is not artinian as a right $S$-module, so we must have $\Kdim_S(\theta(X)/\theta(Y)) \geq 1$. Now applying \cite[Lemma 2.3]{ardakov-Kdim}, we see that $n = \Kdim(R) + 1 \leq \Kdim(S)$.

Together with the inequality of \cite[Corollary 2.9(iv)]{wang-quantum}, we see that
$$n \leq \Kdim(S) \leq \Kdim(\gr(R)) + 1 = n,$$
and the result follows.
\end{proof}

\subsection{Classical Krull dimension}

\begin{thm}\label{thm: clKdim}
Let $k$ be a division ring, and $R\in \RSPS^n(k)$. Then $\clKdim(R) = n$.
\end{thm}

\begin{proof}
It is well known that $\clKdim(R) \leq \Kdim(R)$ (e.g. \cite[Lemma 6.4.5]{MR}), and here $\Kdim(R) = n$ by Theorem \ref{thm: Kdim}. To obtain a lower bound, consider the chain of ideals $\{\mathfrak{m}_i R\}_{i=0}^n$ given in Proposition \ref{propn: quotient extensions}. This chain has length $n$, and again by Proposition \ref{propn: quotient extensions}, the quotient rings are iterated local skew power series rings over $k$, and are hence prime by Proposition \ref{propn: primality}. Hence $\clKdim(R)\geq n$.
\end{proof}

\begin{rk}
Rigidity is a sufficient, but not a necessary, condition to have $\clKdim(R) = \Kdim(R)$, as shown by the following proposition together with Non-example \ref{nonex: pure aut}.
\end{rk}

\begin{propn}\label{propn: clKdim}
Let $k$ be a division ring, and let $R \in\SPS^n(k)$ have \emph{pure automorphic} type. Then $\clKdim(R) = n$.
\end{propn}

\begin{proof}
Suppose $R = k[[x_1;\sigma_1]][[x_2; \sigma_2]] \dots [[x_n;\sigma_n]]$. Note that $(x_n)$ is a two-sided ideal of $R$, and
$$R/(x_n) \cong k[[x_1;\sigma_1]][[x_2; \sigma_2]] \dots [[x_{n-1};\sigma_{n-1}]].$$
The result follows by induction on $n$.
\end{proof}

It is likely that, for many $R\in\SPS^n(k)$, we have $\clKdim(R) < n$. Of particular interest is the case where $k = \mathbb{F}_p$ or $k = \mathbb{Z}_p$, and $R=kG$ is the completed group algebra of a soluble $p$-valuable group $G$ over $k$. Based on the corresponding result for polycyclic group algebras \cite[Theorem H1]{roseblade}, we can ask the following question:

\begin{qn}
Is $\clKdim(\mathbb{F}_pG)$ equal to the plinth length of $G$?
\end{qn}

By Example \ref{ex: rigid soluble IAs}, this is true in the case when $G$ is supersoluble.

\subsection{Global dimension}

\begin{thm}\label{thm: Exts}
Let $R$ be a complete local ring and $S\in\SPS(R)$. Then $\Ext_R^i(M,R) \cong \Ext_S^{i+1}(M,S)$ as abelian groups for all $i$ and all left $S$-modules $M$ that are finitely generated as left $R$-modules.
\end{thm}

\begin{proof}
This follows from a similar argument to \cite{rinehart-rosenberg}, after minor adjustments -- see e.g. \cite[Theorem 3.1]{schneider-venjakob-codim}.
\end{proof}

\begin{cor}
Let $k$ be a field, $R\in\SPS^n(k)$, and $S = R[[x;\sigma,\delta]]\in\SPS(R)$. Then $\gldim(S) = \gldim(R) + 1$.
\end{cor}

\begin{proof}
We have augmentation maps $R\to k$ and $S\to k$, so by \cite[Corollary 3.7]{brumer}, we have $\gldim(R) = \pd_R(k)$ and $\gldim(S) = \pd_S(k)$.

Suppose that $\pd_R(k) = n < \infty$. Then by \cite[VI, Ex. 9]{cartan-eilenberg}, we have $\Ext_R^n(k,R)\neq 0$, and hence by Theorem \ref{thm: Exts}, $\Ext_S^{n+1}(k,S)\neq 0$, so that $\pd_S(k)\geq n+1$.

Now assume (for contradiction) that $\pd_S(k)\geq n+2$. By the same argument, this gives us $\Ext_S^{n+2}(k,S)\neq 0$, and hence $\Ext_R^{n+1}(k,R)\neq 0$. But this implies that $\pd_R(k) \geq n+1$, a contradiction. So we must have $\pd_S(k) = n+1 = \pd_R(k) + 1$.

Finally, as $\pd_k(k) = 0$, it is easy to see by induction on the rank of $R$ (as an iterated local skew power series ring over $k$) that we always have $\pd_R(k) < \infty$.
\end{proof}

\begin{cor}\label{cor: gldim}
Let $k$ be a field and $R\in\SPS^n(k)$. Then $\gldim(R) = n$.\qed
\end{cor}

\bibliography{../research/biblio/biblio-all}
\bibliographystyle{plain}

\end{document}